\documentclass[11pt]{article}
\date{}

\usepackage[title]{appendix}
\usepackage{xcolor}
\usepackage[margin=1in]{geometry}                % See geometry.pdf to learn the layout options. There are lots.
\usepackage{graphicx}
\usepackage{subcaption}
\usepackage{pdflscape}
\usepackage{amssymb}
\usepackage[normalem]{ulem}
\usepackage{hyperref}
\usepackage{enumitem}
\usepackage{mathrsfs}
\usepackage{epstopdf}
\usepackage{rotating}
\usepackage{longtable} % for 'longtable' environment
\usepackage{adjustbox}
\usepackage{float}

\usepackage{color}
\usepackage{bbm, dsfont}
\usepackage{pst-node}
\usepackage{tikz-cd}
\usepackage{amsfonts} %mathbb{}
\usepackage{geometry}
\usepackage{amsthm}
\usepackage{amsmath}
\usepackage{titlesec}

\usepackage{amssymb}
\usepackage{enumitem}
\usepackage{float}
\usepackage [english]{babel}
\usepackage [autostyle, english = american]{csquotes}
\usepackage{algorithm}
\usepackage[noend]{algpseudocode} %pseudocode
\makeatletter
\def\BState{\State\hskip-\ALG@thistlm}
\makeatother
\usepackage{hyperref}
\usepackage[normalem]{ulem}

\newlist{casess}{enumerate}{1}
\setlist[casess]{label=     \textbf{Case} \arabic*:}
\usepackage{mathtools}

\makeatletter
\newcommand*{\rom}[1]{\expandafter\@slowromancap\romannumeral #1@}
\makeatother

\usepackage{etoolbox}

\makeatletter
\patchcmd{\ttlh@hang}{\parindent\z@}{\parindent\z@\leavevmode}{}{}
\patchcmd{\ttlh@hang}{\noindent}{}{}{}
\makeatother

\usepackage{listings}
\usepackage{color} %red, green, blue, yellow, cyan, magenta, black, white
\definecolor{mygreen}{RGB}{28,172,0} % color values Red, Green, Blue
\definecolor{mylilas}{RGB}{170,55,241}

\newlist{Assumptions}{enumerate}{1}
\setlist[Assumptions]{label=     \textbf{Assumption} \arabic*:}

\makeatletter

\newsavebox{\@brx}
\newcommand{\llangle}[1][]{\savebox{\@brx}{\(\m@th{#1\langle}\)}%
  \mathopen{\copy\@brx\kern-0.5\wd\@brx\usebox{\@brx}}}
\newcommand{\rrangle}[1][]{\savebox{\@brx}{\(\m@th{#1\rangle}\)}%
  \mathclose{\copy\@brx\kern-0.5\wd\@brx\usebox{\@brx}}}
\makeatother

\usepackage{lipsum} % for generating filler text
\usepackage{titlesec}
\titleformat{\subsection}[runin]% runin puts it in the same paragraph
       {\normalfont\bfseries}% formatting commands to apply to the whole heading
       {\thesubsection}% the label and number
       {0.5em}% space between label/number and subsection title
       {}% formatting commands applied just to subsection title
       [.]% punctuation or other commands following subsection title

\newcommand{\A}{\mathfrak{A}}

\newcommand{\B}{\mathfrak{B}} 

\newcommand{\CC}{\mathbb{C}}

\usepackage{tikz-cd}
\usepackage{graphicx}

\def\N{\mathbb{N}}

\def\R{\mathbb{R}}

\def\X{\mathcal{X}}

\def\d{{\rm d}}

\def\({\left(}
\def\[{\left[}
\def\){\right)}
\def\]{\right]}

\def\Si{\Sigma}

\def\<{\langle}
\def\>{\rangle}
\providecommand{\norm}[1]{\lVert#1\rVert}

 \newtheorem{thm}{Theorem}[section]
 \newtheorem{cor}[thm]{Corollary}
 
 \newtheorem{lem}[thm]{Lemma}
 
 \theoremstyle{definition}
 \newtheorem{defn}[thm]{Definition}
 \theoremstyle{remark}
 \newtheorem{rem}[thm]{Remark}
 \newtheorem{ex}[thm]{Example}
 \numberwithin{equation}{section}

\numberwithin{equation}{section}

\usepackage[backend=biber, maxnames=10]{biblatex}
\begin{document}

%-------------------------------------------------------------------------------------------------------
% Title
%-------------------------------------------------------------------------------------------------------

\title{On the continuity of derivations over locally regular Banach algebras}

\author{Felipe I. Flores
\footnote{
\textbf{2020 Mathematics Subject Classification:} Primary 46H40, Secondary 46H05, 43A15.
\newline
\textbf{Key Words:} Automatic continuity, derivations, regular Banach function algebra, Banach bimodule, $L^p$-crossed product, finite codimension. }
}

\maketitle

%-------------------------------------------------------------------------------------------------------
 %Abstract
%-------------------------------------------------------------------------------------------------------

\begin{abstract}
We study the problem of continuity of derivations over Banach algebras. More specifically, we consider derivations over a class of Banach algebras that contain dense `$C^*$-like' subalgebras. The results we prove are then applied to $L^p$-crossed products and symmetrized $L^p$-crossed products. For example, it follows that every derivation over the $L^p$-crossed product $F^p(G,X,\alpha)$ is continuous, provided that $G$ is infinite, finitely generated, has polynomial growth, and acts freely on the compact Hausdorff space $X$.
\end{abstract}

%\tableofcontents

%-------------------------------------------------------------------------------------------------------
\section{Introduction}\label{introduction}
%-------------------------------------------------------------------------------------------------------

Let $\B$ be a Banach algebra and let $\X$ be a Banach $\B$-bimodule. A derivation is a linear map $D: \B \to \X$ that satisfies the Leibniz rule:
$$
D(ab) = D(a)b + aD(b),\quad\text{ for all }a,b\in\B.
$$

The study of derivations on Banach algebras is a central topic in Banach algebra theory. Indeed, many structural properties of the algebra are expressed using derivations. Some examples of this are given by amenability, weak amenability, or Hochschild cohomology (see \cite{Ru20}). The study of automatic continuity for derivations also serves as a base case for the general theory of automatic continuity (see \cite{Da00}).

Indeed, a central question about derivations is whether they are automatically continuous. It is known, due to Ringrose \cite{Ri72}, that every derivation over a $C^*$-algebra is automatically continuous. However, more general results are hard to obtain. In fact, it is an open problem to characterize the groups $G$ such that all derivations of $L^1(G)$ are continuous \cite[Question 22]{Da78}. Positive results in this direction were obtained for abelian and compact groups by Jewell \cite{Je77}, for `factorizable' groups by Willis \cite{Wi92}, and for groups of polynomial growth by Runde \cite{Ru96}. We remark that all of these results use fairly different sorts of techniques, so the question is very interesting.

In \cite{flores2024}, the author extended Runde's approach to provide a large class of (generalized) convolution algebras of the form $L^1_\alpha(G,\A)$ with the property that every derivation over them is continuous. These algebras are associated with compactly generated groups of polynomial growth. In \cite{flores2025notefinitedimensionalquotientsproblem}, however, the condition of compact generation was lifted for many of these examples.

In the more general setting, the technique introduced by Ringrose was successfully extended by Bade and Curtis \cite{BaCu74} and Longo \cite{Lo79} to produce results for abstract Banach algebras. 

The specific case where $\X=\B$ has received special attention (see \cite{Sa60,Jo69}). But it seems to be mostly solved by Johnson and Sinclair \cite{JoSi68}, who proved that every derivation $D:\B\to\B$ is continuous if $\B$ is semisimple. We remark, however, that the $L^p$-crossed product algebras introduced and studied in the last section of the paper are not known to be semisimple, except for the trivial cases (that is, when $p\in\{1,2,\infty\}$), so the results in \cite{JoSi68} do not cover them. 

The present article aims to extend Longo's \cite{Lo79} approach to the automatic continuity of derivations in order to obtain new examples of algebras $\B$ with the property that every derivation over $\B$ is continuous. In fact, in doing so, we also extend some of the results obtained in \cite{flores2024} in a manner that will be explained in the next paragraphs. An interesting fact about the algebras we study is that we do not demand them to be involutive, but we do require a suitable `$C^*$-like' dense subalgebra. The key observation here is that Longo requires an everywhere defined, very strong functional calculus (see condition FC in \cite[pag. 22]{Lo79}), but the works \cite{Ru96,flores2024} prove that one can obtain results about automatic continuity just with a densely defined smooth functional calculus. This motivates us to introduce the notion of `locally regular' inclusions $\A\subset \B$ (see Definition \ref{main-def} and Remark \ref{main-rem} for the historical motivation). The idea is to require $\A$ to be an $A^*$-algebra where every element $b\in\A$ generates a regular Banach function algebra, so that all of the arguments can be carried over in a dense manner through $\A$.

Our first main theorem of the article provides sufficient conditions (on $\A$) such that every derivation of $\B$ is continuous.

\begin{thm}[see Theorem \ref{main-simple}]\label{main-simple-intro}
    Let $\A\subset\B$ be a locally regular inclusion and let $D:\B\to\X$ be a derivation into a Banach $\B$-bimodule. Further suppose that $\A$ is unital and that ${\rm C^*}(\A)$ has no proper closed two-sided ideals of finite codimension. Then $D$ is continuous.
\end{thm}

Our techniques also suffice to prove the following theorem, which is about derivations where the target module is a special completion of the algebra in question. In particular, they apply to some notable cases, such as $\X=\B$ or $\X={\rm C^*}(\B)$, provided that $\B$ is an $A^*$-algebra. Note that the theorem above requires $\B$ to be unital (and to share the unit of $\A$), whereas the following theorem does not.

\begin{thm}[see Theorem \ref{main}]\label{main-intro}
Let $\A$ be an $A^*$-algebra and let $\B_1,\B_2$ be Banach algebras such that there are continuous, injective homomorphisms $\A\to \B_i$ and $\B_i\to {\rm C^*}(\A)$ that make the diagrams 
$$
\begin{tikzcd}
	&{\B_i} \\
    {\A} && {{\rm C^*}(\A)} 
	\arrow["", hook, from=1-2, to=2-3]
	\arrow["", hook, from=2-1, to=1-2]
    \arrow["", hook, from=2-1, to=2-3]
\end{tikzcd}
$$
commute. Furthermore, suppose that $\A\to\B_1$ is a locally regular inclusion. Then every derivation $D:\B_1\to \B_2$ is continuous.
\end{thm}

The above-mentioned theorems have a wide range of applications. For example, one could take $\B_1=L^1(G\mid \mathscr C)$, the algebra of integrable cross-sections on the Fell bundle $\mathscr C\to G$. Then the results in \cite[Proposition 4.10]{flores2024} provide us with a locally regular inclusion $\mathfrak E\subset L^1(G\mid \mathscr C)$ as soon as $G$ is locally finite or compactly generated and has polynomial growth. In this case, an application of Theorem \ref{main-simple-intro} yields a generalization of Theorem \cite[Theorem 4.16]{flores2024}, as the result no longer requires the assumption of symmetry.

In the same setting, Theorem \ref{main-intro} implies that every derivation $D: L^1(G\mid \mathscr C)\to {\rm C^*}(G\mid \mathscr C)$ is continuous. A result of this generality is unachievable with the methods in \cite{flores2024}.

The applications that we want to discuss in more detail come from the theory of $L^p$-operator algebras, so we will mention them now. For the general theory of $L^p$-operator algebras and $L^p$-crossed products, we recommend \cite{CP12,CP13,Ga21}. Let us also mention the recent developments \cite{BlCP19,ChGaTh24,De25}. Symmetrized $L^p$-crossed products have the advantage of possessing an involution. These have been studied in \cite{CP19,AuOr22,El25,BaKw25}, often under the name of `algebras of pseudofunctions'. In what follows, $F^p(G,X,\alpha)$ and $F^p_*(G,X,\alpha)$ will denote, respectively, the $L^p$-crossed product and the symmetrized $L^p$-crossed product associated with a continuous action $G\overset{\alpha}{\curvearrowright} X$ and a parameter $p\in[1,\infty]$.

\begin{cor}[see Corollary \ref{cor2}]
    Let $p\in[1,\infty]$ and let $G$ be a countably infinite group that is either locally finite or finitely generated and of polynomial growth. Let $\alpha$ be a continuous action of $G$ on the compact Hausdorff space $X$ such that the restriction of $\alpha$ to any closed invariant subset is topologically free. Let $\X$ be a Banach $F^p(G,X,\alpha)$-bimodule. Then every derivation $D:F^p(G,X,\alpha)\to \X$ is continuous.
\end{cor}

The theorem above remains valid if one wishes to change $F^p(G,X,\alpha)$ by its symmetrized counterpart $F^q_*(G,X,\alpha)$ (maintaining the assumptions on $(G,X,\alpha)$). However, $F^q_*(G,X,\alpha)$ is also an $A^*$-algebra, so we can also apply Theorem \ref{main-intro} to it. Doing so yields the following.

\begin{cor}[see Corollary \ref{cor1}]
Let $p,q\in[1,2]$ and let $G$ be either a countable locally finite group or a compactly generated locally compact group with polynomial growth. Let $\alpha$ be a continuous action of $G$ on the locally compact Hausdorff space $X$. Then every derivation $D: F^p_*(G,X,\alpha)\to F^q_*(G,X,\alpha)$ is continuous.
\end{cor}

The result above seems particularly interesting in the cases where $p=1$, $q=2$, and $p=q$. That is, for derivations of the form $L^1_\alpha(G,C_0(X)) \to F^q_*(G,X,\alpha)$, $F^p_*(G,X,\alpha)\to C_0(X)\rtimes_{\alpha}G$, or $F^p_*(G,X,\alpha)\to F^p_*(G,X,\alpha)$. 

It was shown by Gardella and Thiel that the $L^p$-operator algebras associated with (finite or infinite) cyclic groups are semisimple \cite{GaTh15}. Hence, the result of Johnson and Sinclair \cite{JoSi68} guarantees the automatic continuity of derivations of the form $F^p(G)\to F^p(G)$ for these groups. For the moment, it is unclear whether more general $L^p$-crossed products are semisimple.

By the end of the paper, the reader will notice that our main results, Theorem \ref{main-simple-intro} and Theorem \ref{main-intro}, can also be applied to other `crossed product-like' algebras, such as those introduced by Dirksen, de Jeu and Wortel \cite{DdJW11}. Moreover, thanks to the generality of the results in \cite{Fl24}, one could also consider twisted $C^*$-dynamical systems and obtain, with no extra effort, results for the algebras of a twisted crossed product form (see \cite{DeFaPa25}), including twisted $L^p$-crossed products and twisted symmetrized $L^p$-crossed products. We will, however, not explicitly state these results here.

%-------------------------------------------------------------------------------------------------------
\section{Main results}\label{AC}
%-------------------------------------------------------------------------------------------------------

If $\B$ is a Banach algebra, $\B(b_1,\ldots, b_n)$ denotes the closed subalgebra of $\B$ generated by the elements $b_1,\ldots, b_n\in\B$. We set $\widetilde{\B}$ to denote the smallest unitization of $\B$ (it coincides with $\B$ when $\B$ is already unital). As usual, if $\B$ is not unital, the norm of an element $b+\lambda 1\in \widetilde{\B}$ is given by $\|b+\lambda1\|_{\widetilde{\B}}=\|b\|_\B+|\lambda|$, and we write
$$
{\rm Spec}_{\B}(b)=\{\lambda\in\CC\mid b-\lambda1\textup{ is not invertible in }\widetilde\B\}
$$ 
to denote the spectrum of an element $b\in\B$.

If $\B$ has an involution, $\B_{\rm sa}$ denotes the set of self-adjoint elements in $\B$, that is, of all $b\in\B$ such that $b^*=b$. A Banach $^*$-algebra admitting a $C^*$-norm is called an $A^*$-algebra (also called a \emph{reduced} Banach $^*$-algebra). If $\B$ is an $A^*$-algebra, we use ${\rm C^*}(\B)$ to denote the universal $C^*$-algebra generated by $\B$. In such a case, $\B$ can be identified with a dense subalgebra of ${\rm C^*}(\B)$, and the embedding is contractive.

If $\B$ is a commutative Banach algebra with spectrum $\Delta_\B$, then $\hat b\in C_0(\Delta_\B)$ denotes the Gelfand transform of $b\in \B$. If the Gelfand transform is injective, $\B$ is called a Banach function algebra. 

\begin{defn}
    Let $\B$ be a Banach function algebra with spectrum $\Delta_\B$. $\B$ is called {\it regular} if, for every closed set $X\subset \Delta_\B$ and every point $\omega\in \Delta_\B\setminus X$, there exists an element $b\in \B$ such that $\hat b(\varphi)=0$ for all $\varphi\in X$ and $\hat b(\omega)\not=0$.
\end{defn}

\begin{defn}\label{main-def}
    Let $\A$ be an $A^*$-algebra and let $\B$ be a Banach algebra. A continuous monomorphism with dense image $\A\to\B$ is called a \emph{locally regular inclusion} if \begin{equation}\label{eqqq}
        \A(b)\text{ is a regular Banach function algebra, }\forall b\in\A_{\rm sa}.
    \end{equation}
\end{defn}

In what follows, we will identify a given Banach algebra with any of its continuous monomorphic images. Under such a view, a subset $\A$ of the Banach algebra $\B$ produces a locally regular inclusion if $\A$ is a dense subalgebra that also admits an involution and a finer norm under which it becomes a Banach $^*$-algebra satisfying \eqref{eqqq}.

\begin{rem}\label{main-rem}
    The terminology used in Definition \ref{main-def} is inspired by the terminology of B. Barnes. In \cite[Definition 4.1]{Ba81}, Barnes defines local regularity for an $A^*$-algebra $\A$ by requiring that $\A(a)$ is a regular Banach function algebra, for all $a$ in a dense subset of $\A_{\rm sa}$. When restricted to the setting of $A^*$-algebras, Barnes' definition is weaker than ours, but it does not seem to generalize in a natural manner to Banach algebras without involutions. 
\end{rem}

Let $\B$ be a Banach algebra. A Banach space $\mathcal X$ that is also a $\B$-bimodule is called a \emph{Banach $\B$-bimodule} if the maps 
$$
\B\times \mathcal X\ni(b,\xi)\mapsto b\xi \in\mathcal X \quad\text{ and }\quad \X\times\B\ni(\xi,b)\mapsto \xi b \in\mathcal X
$$ 
are jointly continuous. 

\begin{ex}
    Let $\B\subset \mathfrak C$ be an inclusion of Banach algebras. Then $\mathfrak C$ is naturally a Banach $\B$-bimodule under the actions 
    $$ \B\times\mathfrak C\ni (b,c)\mapsto bc\in \mathfrak C\quad\text{ and }\quad \mathfrak C\times\B\ni (c,b)\mapsto cb\in \mathfrak C.$$
\end{ex}

\begin{defn}
    Let $\B$ be a Banach algebra and let $\X$ be a Banach $\B$-bimodule. A linear map $D:\B\to\X$ is called a \emph{derivation} if, for each $a,b\in\B$, one has $$D(ab)=D(a)b+aD(b).$$
\end{defn}

\begin{defn}
    Let $\B$ be a Banach algebra and let $D:\B\to\X$ be a derivation. Then 
    $$
    \mathscr I(D)=\{b\in\B\mid \text{the maps } a\mapsto D(ba)\text{ and } a\mapsto D(a b)\text{ are continuous}\}
    $$
    is the \emph{continuity ideal} of $D$.
\end{defn}

The following lemma is an immediate application of well-known results by Barnes that can be found in \cite[Theorem 2.1, Theorem 2.2]{Ba87}.

\begin{lem}\label{barnesss}
    Let $\A\subset\B$ be a locally regular inclusion and let $J$ be a closed two-sided ideal of ${\rm C^*}(\A)$. Then 
    $$
    {\rm Spec}_{\A/(J\cap\A)}(a+J\cap\A)={\rm Spec}_{{\rm C^*}(\A)/J}(a+J),
    $$
    for all $a\in\A_{\rm sa}$.
\end{lem}

\begin{lem}\label{DR}
    Let $\A\subset\B$ be a locally regular inclusion and let $D:\B\to\X$ be a derivation into a Banach $\B$-bimodule. Then the ideal $\overline{\mathscr I(D)\cap \A}^{{\rm C^*}(\A)}\cap \A$ is closed and has finite codimension in $\A$.
\end{lem}
\begin{proof}
    For simplicity, set $J=\overline{\mathscr I(D)\cap \A}^{{\rm C^*}(\A)}\subset {\rm C^*}(\A)$ and $K= J\cap \A\subset \A$. Note that $K$ is a closed $^*$-ideal. Let $\pi:{\rm C^*}(\A)\to {\rm C^*}(\A)/J$ and $\pi': \A\to \A/K$ be the quotient maps and $\iota: \A\to {\rm C^*}(\A)$ the canonical inclusion. Then there exists a continuous $^*$-homomorphism $\varphi: \A/K\to {\rm C^*}(\A)/J$ that makes the following diagram commute
    $$
    \begin{tikzcd}
[arrows=rightarrow]
    \A\arrow{d}{\pi'}\arrow{r}{\iota} & {\rm C^*}(\A)
    \arrow{d}{\pi}
    \\
    \A/K
    \arrow{r}{\varphi}
    & {\rm C^*}(\A)/J.
\end{tikzcd}
$$
Since $\varphi$ maps $\A/K$ injectively into the $C^*$-algebra ${\rm C^*}(\A)/J$, it is necessary for $\A/K$ to be an $A^*$-algebra. 

Now, for the sake of contradiction, we will assume that $\A/K$ is infinite-dimensional. That makes $\A/K$ an infinite-dimensional $A^*$-algebra and so, by \cite[Corollary 5.4.3]{Au91}, there exists $a\in\A_{\rm sa}$ such that the spectrum $\Si={\rm Spec}_{\A/K}(\pi'(a))$ is infinite and we have $\Si={\rm Spec}_{{\rm C^*}(\A)/J}(\pi(a))\subset \R$, due to Lemma \ref{barnesss}. Using the regularity of $\A(a)$, we can find a sequence of functions $f_n\in C(\R)$ that are nonzero in $\Si$, have disjoint supports, and such that the sequence $a_n=f_n(a)$ belongs to $\A$ and satisfies 
$$
a_na_m=f_n(a)f_m(a)=(f_n\cdot f_m)(a)=0,\quad \text{ for }n\not=m.
$$
It also satisfies
$$
a_n^2\not\in K,\quad \text{ for all }n\in\N,
$$
since 
$$
\pi'(a_n^2)=\pi'(f_n^2(a))=\varphi^{-1}(f_n^2(\pi(a)))\not=0.
$$ 
Replacing the $a_n$'s by appropriate scalar multiples, we can assume that $\norm{a_n}_\B\leq 1$. Note that $a_n^2\not \in \mathscr I(D)$, as the opposite situation would mean that $a_n^2\in  \mathscr I(D)\cap \A\subset K$. Thus, one can find $b_n\in \B$ such that 
    $$
    \norm{b_n}_\B\leq 2^{-n}\quad\text{ and }\quad \norm{D(a^2_nb_b)}_\X\geq M \norm{D(a_n)}_\X+n,
    $$
    where $M$ satisfies $\norm{b\xi}_\X\leq M \norm{b}_\B\norm{\xi}_\X$. Now consider the element $c=\sum_{n\in\N} a_nb_n\in\B$. It satisfies $\norm{c}_\B\leq 1$ and $a_nc=a_n^2b_n$, for all $n\in\N$. Hence 
    \begin{align*}
        M\norm{D(c)}_\X\geq\norm{a_nD(c)}_\X=\norm{D(a_nc)-D(a_n)c}_\X \geq \norm{D(a_n^2b_n)}_\X-M\norm{D(a_n)}_\X\geq n,
    \end{align*}
    which is absurd. Hence, the quotient $\A/K$ has to be finite-dimensional.
\end{proof}

\begin{thm}\label{main-simple}
    Let $\A\subset\B$ be a locally regular inclusion and let $D:\B\to\X$ be a derivation into a Banach $\B$-bimodule. Further suppose that $\A$ is unital and that ${\rm C^*}(\A)$ has no proper closed ideals of finite codimension. Then $D$ is continuous.
\end{thm}
\begin{proof}
    An application of Lemma \ref{DR} tells us that the ideal $\overline{\mathscr I(D)\cap \A}^{{\rm C^*}(\A)}\cap \A$ has finite codimension in $\A$. From this, it is easy to observe that $\overline{\mathscr I(D)\cap \A}^{{\rm C^*}(\A)}$ has finite codimension in ${\rm C^*}(\A)$, and because of the assumptions, it follows that $1\in \overline{\mathscr I(D)\cap \A}^{{\rm C^*}(\A)}$. Appealing to basic spectral theory, we note that $\mathscr I(D)\cap \A$ must contain an element $a$ that is invertible in ${\rm C^*}(\A)$. However, by invoking Lemma \ref{barnesss}, we see that $a^{-1}$ necessarily lies in $\A$. Therefore, $1\in \mathscr I(D)$ and the map 
    $$
    \B\ni b\mapsto D(b1)=D(b)\in\X
    $$
    is continuous.
\end{proof}

\begin{rem}\label{conditions}
The above-proved theorem should be compared with \cite[Theorem 3.6]{flores2024}. While \cite[Theorem 3.6]{flores2024} applies to a larger class of maps (and not just derivations), it has more assumptions, making it more restrictive in the context of derivations into Banach bimodules. The most obvious difference is that $\B$ is no longer required to have an involution. Hence, when restricted to our setting, the theorem just presented is far more general.
\end{rem}

The following lemma is an immediate application of a classic result of Barnes on the uniqueness of $C^*$-norms. See, for example, \cite[Theorem 4.2]{Ba81}.

\begin{lem}\label{iid} Let $\A\subset\B$ be a locally regular inclusion. Let $\{0\}\not=I\subset {\rm C^*}(\A)$ be a closed two-sided ideal. Then $I\cap \A\not=\{0\}$.
\end{lem}

\begin{thm}\label{main}
Let $\A$ be an $A^*$-algebra and let $\B_1,\B_2$ be Banach algebras such that there are continuous, injective homomorphisms $\A\to \B_i$ and $\B_i\to {\rm C^*}(\A)$ that make the diagrams 
$$
\begin{tikzcd}
	&{\B_i} \\
    {\A} && {{\rm C^*}(\A)} 
	\arrow["", hook, from=1-2, to=2-3]
	\arrow["", hook, from=2-1, to=1-2]
    \arrow["", hook, from=2-1, to=2-3]
\end{tikzcd}
$$
commute. Furthermore, suppose that $\A\to\B_1$ is a locally regular inclusion. Then every derivation $D:\B_1\to \B_2$ is continuous.
\end{thm}

\begin{proof}
    Again, for the sake of simplicity, set $J=\overline{\mathscr I(D)\cap \A}^{{\rm C^*}(\A)}$ and $K=J\cap\A$. Because of Lemma \ref{DR}, we have that $K$ has finite codimension in $\A$ and $J$ has finite codimension in ${\rm C^*}(\A)$. Now consider the set 
    $$
    S=\{c\in\B_2\mid ct=0,\,\forall t\in \mathscr I(D)\}.
    $$
    If we show that $S=\{0\}$, then the theorem is proven, and the reason is the following. Suppose that $S=\{0\}$ and, by means of the closed graph theorem, let us show that $D$ is continuous. In order to achieve that purpose, take a sequence $a_n\in \B_1$ and an element $c\in\B_2$ such that 
    $$
    \norm{a_n}_{\B_1}\to 0\quad \text{ and }\quad \norm{D(a_n)-c}_{\B_2}\to 0.
    $$
    Observe that $D(a_nt)\to 0$ for all $t\in \mathscr I(D)$. On the other hand, one sees that 
    $$
    D(a_nt)=a_nD(t)+D(a_n)t\to ct.
    $$
    Hence $ct=0$, for all $t\in \mathscr I(D)$. It follows that $c\in S=\{0\}$, which proves the continuity of $D$.
    
      Now, it remains to show that $S=\{0\}$. In order to do so, first note that $S$ is a two-sided ideal of $\B_2$. Note also that $I=\overline{S}^{{\rm C^*}(\A)}$ and $J$ are two-sided ideals of ${\rm C^*}(\A)$ and hence $C^*$-algebras; in particular, they contain bounded approximate identities. It follows that $IJ=I\cap J$ (one can appeal to \cite[Theorem 11.10]{BD73}). But then, one can easily check that $IJ=\{0\}$ and hence 
    $$
    (I\cap \A)\cap K\subset  I\cap J=\{0\}.
    $$
    Since $K$ has finite codimension, this forces $I\cap \A$ to be finite-dimensional. But $I\cap \A$ is a closed $^*$-ideal of $\A$ and therefore a finite-dimensional $A^*$-algebra. 
    
    In order to achieve a contradiction, suppose now that $I\cap \A\not=\{0\}$. Then, up to renorming, $I\cap \A$ must be a nontrivial finite-dimensional $C^*$-algebra and so it must contain a unit $p\in I\cap \A$. It is easy to see that $p$ is a central element in $\A$ and that it satisfies $I\cap \A=\A p$. Using this, we note that the map 
    $$
    \A\ni a\mapsto ap\in I\cap \A
    $$ 
    has finite-dimensional range and factorizes the map 
    $$
    \B_1\ni a\mapsto D(ap)\in \B_2,
    $$ 
    making the latter continuous, so we have $p\in \mathscr I(D)\cap \A\subset K$. It then follows that $p\in (I\cap \A)\cap K=\{0\}$, a contradiction. Hence $I\cap \A=\{0\}$.

    Finally, Lemma \ref{iid} allows us to conclude that $S\subset I=\{0\}$ from the fact that $I\cap \A=\{0\}$, so the theorem is proven. \end{proof}

%-------------------------------------------------------------------------------------------------------
\section{Applications to \texorpdfstring{$L^p$}--crossed products}\label{applications}
%-------------------------------------------------------------------------------------------------------

As mentioned before, the purpose of this section is to apply the previous results to $L^p$-crossed products and symmetrized $L^p$-crossed products. For that purpose, we will fix a $C^*$-dynamical system $(G,C_0(X),\alpha)$, where $G$ is an amenable locally compact group, $X$ is a locally compact Hausdorff space, and $\alpha$ is a continuous action of $G$ on $X$. In what follows, we consider that $G$ is always endowed with a Haar measure $\d x$.

We shall abuse the notation and call $\alpha$ both the action on $X$ and the induced action on $C_0(X)$. That is
$$
\alpha_x(f)(z)=f\big(\alpha_{x^{-1}}(z)),
$$
where $x\in G,z\in X, f\in C_0(X)$.

Then the generalized convolution algebra $L^1_{\alpha}(G,C_0(X))$ is given by all Bochner integrable functions $\Phi:G\to C_0(X)$, endowed with the product 
\begin{equation*}
    \Phi*\Psi(x)=\int_G \Phi(y)\alpha_y[\Psi(y^{-1}x)]\d y
\end{equation*} 
and the involution 
\begin{equation*}
    \Phi^*(x)=\Delta(x^{-1})\alpha_x[\Phi(x^{-1})^*].
\end{equation*} 
With these operations and the norm $\norm{\Phi}_{1}=\int_G\norm{\Phi(x)}_{C_0(X)}\d x$, we obtain a Banach $^*$-algebra structure for $L^1_{\alpha}(G,C_0(X))$.

It is well-known that the algebra $L^1_{\alpha}(G,C_0(X))$ admits $C^*$-norms, so it is an $A^*$-algebra. Its enveloping $C^*$-algebra is denoted $C_0(X)\rtimes_{\alpha}G$ and we call it the \emph{crossed product} of $C_0(X)$ by $G$. We now proceed to introduce $L^p$-crossed products.

\begin{defn}\label{D-CovRep}
Let $p \in [1, \infty],$ and let $(G,C_0(X),\alpha)$ be a $C^*$-dynamical system as above. Let $(Y,\mathcal B,\mu)$ be a measure space. A \emph{covariant representation} of $(G,C_0(X),\alpha)$ on $L^p (Y, \mu)$ is a pair $(v, \pi)$ consisting of a map $g \mapsto v_g$ from $G$ to the invertible isometries on $L^p (Y, \mu)$ that is continuous in the sense that $g\mapsto v_g\xi$ is continuous for all $\xi\in L^p(Y,\mu)$, and a non-degenerate, contractive representation $\pi : C_0(X) \to \mathbb B(L^p (Y, \mu)),$ such that
$$
v_xv_y=v_{xy}, \qquad  v_x\pi(f)v_{x}^{-1}=\pi(\alpha_x(f)),
$$
for all $x,y \in G$ and $f \in C_0(X)$.

Given a covariant representation $(v, \pi)$ on $L^p (Y, \mu)$, we define the associated \emph{integrated
representation} $\pi\rtimes v : L^1_{\alpha}(G,C_0(X))\to \mathbb B(L^p (Y, \mu))$ by
$$
(\pi\rtimes v)(\Phi)(\xi)=\int_G \pi(\Phi(x))(v_x(\xi))\d x,  
$$
for all $\Phi\in L^1_{\alpha}(G,C_0(X))$ and all $\xi\in L^p (Y, \mu)$.
\end{defn}

With these definitions at hand, we can introduce the following norm on $L^1_{\alpha}(G,C_0(X))$:
$$
        \|\Phi\|_{p,{\rm max}}=\sup\{\|(\pi\rtimes v)(\Phi)\|\mid (v,\pi) \text{ is a covariant representation on some space }L^p (Y, \mu)\}.
$$
The completion $F^p(G,X,\alpha):=\overline{L^1_{\alpha}(G,C_0(X))}^{\|\cdot\|_{p,{\rm max}}}$ is usually called the full $L^p$-crossed product associated with $(G,C_0(X),\alpha)$. It is easily seen to be a Banach algebra. In the case where the action is trivial, we set $F^p(G):=F^p(G,\{\rm pt\},1)$ and name the resulting algebra the group $L^p$-operator algebra associated with $G$. 

These algebras have appeared in \cite{Ga21,ChGaTh24}. See also \cite{DdJW11,BaKw25,DeFaPa25} for generalizations. For the groupoid case, we recommend \cite{AuOr22}.

It is observed in \cite[Lemma 2.11]{BaKw25} that the formulas 
$$
\pi(f)\xi(x,t)=f(x)\xi(x,t), \quad v_y \xi(x,t)=\xi(\alpha_{y^{-1}}(x),y^{-1}t),
$$
valid for $f\in C_0(X)$, $x,y,t\in G$, and $\xi\in L^p(X\times G)\equiv \ell^p(X,L^p(G))$, define a covariant representation of $L^1_{\alpha}(G,C_0(X))$ on $L^p(X\times G)$. Its integrated form satisfies the formula (note the special notation)
$$
\Lambda_p(\Phi)\xi(x,t)=\int_{G} \Phi(y)(x)\xi(\alpha_{y^{-1}}(x),y^{-1}t)\d y.
$$
Using this representation, one may define the reduced $L^p$-crossed product as the completion $\overline{\Lambda_p\big(L^1_{\alpha}(G,C_0(X))\big)}\subset \mathbb B(L^p(X\times G))$. However, in our setting (the groups are amenable), this construction is known to coincide with the full crossed product (see \cite[Theorem 3.5]{BaKw25}). 

In any case, the representation $\Lambda_p$ allows us to introduce a `symmetrized' version of the $L^p$-crossed product. This construction appears explicitly in \cite{BaKw25}; similar constructions also appear in \cite{AuOr22,DeFaPa25}. In any case, consider the norm 
$$
\norm{\Phi}_{q,*}=\max\{\|\Lambda_p(\Phi)\|,\|\Lambda_q(\Phi)\|\},
$$
where $p,q\in [1,\infty]$ are H\"older duals. We define $F^q_*(G,X,\alpha)$ to be the completion of $L^1_{\alpha}(G,C_0(X))$ under the norm $\norm{\cdot}_{q,*}$. As before, we define the symmetrized group algebra by setting $F^q_*(G):=F^q_*(G,\{\rm pt\},1)$.

The point is that $F^q_*(G,X,\alpha)$ is naturally a Banach $^*$-algebra that sits between $L^1_{\alpha}(G,C_0(X))$ and $C_0(X)\rtimes_{\alpha}G$, so it fits the setting of Theorem \ref{main} perfectly. Indeed, the fact that this algebra has a continuous involution arises from the fact that, for all $\Phi\in L^1_{\alpha}(G,C_0(X)),\xi\in L^p(X\times G),\eta\in L^q(X\times G)$, one has
$$
\langle \Lambda_p(\Phi)\xi,\eta\rangle=\langle\xi,\Lambda_q(\Phi^*)\eta\rangle,
$$
where $\langle\cdot,\cdot\rangle$ denotes the duality pairing between $L^p(X\times G)$ and $L^q(X\times G)$ (cf. \cite{AuOr22}). On the other hand, the inclusions $L^1_{\alpha}(G,C_0(X))\subset F^q_*(G,X,\alpha)\subset C_0(X)\rtimes_{\alpha}G$ follow from complex interpolation. Indeed, note that $F^q_*(G,X,\alpha)$ acts on both $L^p(X\times G)$ and $L^q(X\times G)$, so an immediate application of the Riesz-Thorin interpolation theorem yields the existence of $\theta\in[0,1]$ such that
$$
\norm{\Lambda_2(\Phi)}\leq \norm{\Lambda_p(\Phi)}^{1-\theta}\norm{\Lambda_q(\Phi)}^\theta\leq \norm{\Phi}_{q,*}.
$$

We refer the reader to \cite[Section 3]{AuOr22} or \cite[Sections 2,3]{BaKw25} for a careful discussion of the algebra $F^q_*(G,X,\alpha)$ and the properties that we have just mentioned. The interpolation argument can be explicitly found in \cite{AuOr22} or in \cite{El25}. In \cite{BaKw25}, a different argument is presented.

In any case, let us describe our source of locally regular inclusions that involve $L^p$-crossed products.

\begin{defn}
    A \emph{weight} on the locally compact group $G$ is a measurable, locally bounded function $w: G\to [1,\infty)$ satisfying 
\begin{equation*}\label{submultiplicative}
 w(xy)\leq w(x)w(y)\,,\quad w(x^{-1})=w(x)\,,\quad\forall\,x,y\in G.
\end{equation*} 
\end{defn}

During the rest of the section, $\mathfrak E_w$ will denote the subalgebra of $L^1_{\alpha}(G,C_0(X))$ defined using the weight $w$, as the set of all the functions $\Phi$ such that $\norm{\Phi}_w<\infty$, where the norm in question is given by
$$
\norm{\Phi}_{w}=\max\Big\{{\rm essup}_{x\in  G}\norm{\Phi(x)}_{C_0(X)},\int_G \norm{\Phi(x)}_{C_0(X)} w(x)\d x \Big\}.
$$
It is not hard to see that $\mathfrak E_w$ is a dense Banach $^*$-subalgebra of $L^1_{\alpha}(G,C_0(X))$. Furthermore, the following lemma was proven in \cite[Proposition 4.10]{flores2024}, using some results from \cite{Fl24}. Note that the algebra $\mathfrak E_w$ was first introduced in \cite{Fl24}.

\begin{lem}[\cite{flores2024}]\label{localregularzz}
    Let $G$ be either a countable locally finite group or a compactly generated locally compact group with polynomial growth. Then there exists a weight function $w:G\to[1,\infty)$ such that $\mathfrak E_w(\Phi)$ is a regular Banach function algebra, for every $\Phi= \Phi^*\in\mathfrak E_w$.
\end{lem}

The above lemma implies that, for any Banach algebra completion $\B$ of $\mathfrak E_w$, the inclusion $\mathfrak E_w\subset \B$ is locally regular; thus, our theorems are applicable to $\B$. In particular, the $L^p$-crossed products we just described are Banach algebra completions of $L^1_{\alpha}(G,C_0(X))$; hence, they are completions of $\mathfrak E_w$. The following corollary is then just a simple application of Theorem \ref{main}.

\begin{cor}\label{cor1}
Let $p,q\in[1,2]$ and let $G$ be either a countable locally finite group or a compactly generated locally compact group with polynomial growth. Let $\alpha$ be a continuous action of $G$ on the locally compact Hausdorff space $X$. Then every derivation $D: F^p_*(G,X,\alpha)\to F^q_*(G,X,\alpha)$ is continuous.
\end{cor}

In order to apply Theorem \ref{main-simple}, we also need to guarantee that ${\rm C^*}(\mathfrak E_w)=C_0(X)\rtimes_{\alpha}G$ lacks proper closed two-sided ideals of finite codimension. But this is known to be the case when $G$ is discrete, infinite, and the restriction of $\alpha$ to any invariant subset is topologically free (see \cite[Theorem 29.9]{Ex17}).\footnote{Such actions are often called \emph{residually topologically free}.} As a consequence, we derive the following corollary.

\begin{cor}\label{cor2}
    Let $p\in[1,\infty]$ and let $G$ be a countably infinite group that is either locally finite or finitely generated and of polynomial growth. Let $\alpha$ be a continuous action of $G$ on the compact Hausdorff space $X$ such that the restriction of $\alpha$ to any closed invariant subset is topologically free. Let $\X$ be a Banach $F^p(G,X,\alpha)$-bimodule. Then every derivation $D:F^p(G,X,\alpha)\to \X$ is continuous.
\end{cor}

\section*{Acknowledgments}

The author gratefully acknowledges support from the NSF grant DMS-2144739. The author also wishes to express his gratitude to Professor Ben Hayes for the interesting discussions surrounding this topic. Finally, the author thanks Alonso Delf\'in and the anonymous referee for several comments that improved the exposition in this paper.

\printbibliography

@article {BaCu74,
    AUTHOR = {Bade, W. G. and Curtis, Jr., P. C.},
     TITLE = {The continuity of derivations of {B}anach algebras},
   JOURNAL = {J. Functional Analysis},
  FJOURNAL = {Journal of Functional Analysis},
    VOLUME = {16},
      YEAR = {1974},
     PAGES = {372--387},
}

@article {flores2024,
    AUTHOR = {Flores, F. I.},
     TITLE = {On the continuity of intertwining operators over generalized
              convolution algebras},
   JOURNAL = {J. Math. Anal. Appl.},
  FJOURNAL = {Journal of Mathematical Analysis and Applications},
    VOLUME = {542},
      YEAR = {2025},
    NUMBER = {1},
     PAGES = {Paper No. 128753, 18}
}

@article {Ba81,
    AUTHOR = {Barnes, B. A.},
     TITLE = {Ideal and representation theory of the {$L\sp{1}$}-algebra of
              a group with polynomial growth},
   JOURNAL = {Colloq. Math.},
  FJOURNAL = {Colloquium Mathematicum},
    VOLUME = {45},
      YEAR = {1981},
    NUMBER = {2},
     PAGES = {301--315},
}

@article {Ri72,
    AUTHOR = {Ringrose, J. R.},
     TITLE = {Automatic continuity of derivations of operator algebras},
   JOURNAL = {J. London Math. Soc. (2)},
  FJOURNAL = {Journal of the London Mathematical Society. Second Series},
    VOLUME = {5},
      YEAR = {1972},
     PAGES = {432--438},
}

@article {Ru96,
    AUTHOR = {Runde, V.},
     TITLE = {Intertwining operators over {$L^1(G)$} for {$G\in[{\rm
              PG}]\cap [{\rm SIN}]$}},
   JOURNAL = {Math. Z.},
  FJOURNAL = {Mathematische Zeitschrift},
    VOLUME = {221},
      YEAR = {1996},
    NUMBER = {3},
     PAGES = {495--506},
}

@book {BD73,
    AUTHOR = {Bonsall, F. F. and Duncan, J.},
     TITLE = {Complete normed algebras},
    SERIES = {Ergebnisse der Mathematik und ihrer Grenzgebiete [Results in
              Mathematics and Related Areas]},
    VOLUME = {Band 80},
 PUBLISHER = {Springer-Verlag, New York-Heidelberg},
      YEAR = {1973},
     PAGES = {x+301},
}

@article {Lo79,
    AUTHOR = {Longo, R.},
     TITLE = {Automatic relative boundedness of derivations in {$C\sp{\ast}
              $}-algebras},
   JOURNAL = {J. Functional Analysis},
  FJOURNAL = {Journal of Functional Analysis},
    VOLUME = {34},
      YEAR = {1979},
    NUMBER = {1},
     PAGES = {21--28},
      ISSN = {0022-1236},
}

@article {Fl24,
    AUTHOR = {Flores, F. I.},
     TITLE = {Polynomial growth and functional calculus in algebras of
              integrable cross-sections},
   JOURNAL = {J. Math. Anal. Appl.},
  FJOURNAL = {Journal of Mathematical Analysis and Applications},
    VOLUME = {549},
      YEAR = {2025},
    NUMBER = {2},
     PAGES = {Paper No. 129486, 32},
}

@article {Ba87,
    AUTHOR = {Barnes, B. A.},
     TITLE = {The spectrum of integral operators on {L}ebesgue spaces},
   JOURNAL = {J. Operator Theory},
  FJOURNAL = {Journal of Operator Theory},
    VOLUME = {18},
      YEAR = {1987},
    NUMBER = {1},
     PAGES = {115--132},
}

@book {Au91,
    AUTHOR = {Aupetit, B.},
     TITLE = {A primer on spectral theory},
    SERIES = {Universitext},
 PUBLISHER = {Springer-Verlag, New York},
      YEAR = {1991},
     PAGES = {xii+193},
}

@book {Ex17,
    AUTHOR = {Exel, R.},
     TITLE = {Partial dynamical systems, {F}ell bundles and applications},
    SERIES = {Mathematical Surveys and Monographs},
    VOLUME = {224},
 PUBLISHER = {American Mathematical Society, Providence, RI},
      YEAR = {2017},
     PAGES = {vi+321},
}

@book {Ru20,
    AUTHOR = {Runde, V.},
     TITLE = {Amenable {B}anach algebras},
    SERIES = {Springer Monographs in Mathematics},
      NOTE = {A panorama},
 PUBLISHER = {Springer-Verlag, New York},
      YEAR = {2020},
     PAGES = {xvii+462},
}

@article {JoSi68,
    AUTHOR = {Johnson, B. E. and Sinclair, A. M.},
     TITLE = {Continuity of derivations and a problem of {K}aplansky},
   JOURNAL = {Amer. J. Math.},
  FJOURNAL = {American Journal of Mathematics},
    VOLUME = {90},
      YEAR = {1968},
     PAGES = {1067--1073},
}

@article {Sa60,
    AUTHOR = {Sakai, S.},
     TITLE = {On a conjecture of {K}aplansky},
   JOURNAL = {Tohoku Math. J. (2)},
  FJOURNAL = {The Tohoku Mathematical Journal. Second Series},
    VOLUME = {12},
      YEAR = {1960},
     PAGES = {31--33},
}

@article {Jo69,
    AUTHOR = {Johnson, B. E.},
     TITLE = {Continuity of derivations on commutative algebras},
   JOURNAL = {Amer. J. Math.},
  FJOURNAL = {American Journal of Mathematics},
    VOLUME = {91},
      YEAR = {1969},
     PAGES = {1--10},
}

@misc{CP12,
      title={Analogs of Cuntz algebras on $L^p$ spaces}, 
      author={N. C. Phillips},
      year={2012},
      eprint={1201.4196},
      archivePrefix={arXiv},
      primaryClass={math.FA},
}

@article {BlCP19,
    AUTHOR = {Blecher, D. P. and Phillips, N. C.},
     TITLE = {{$L^p$}-operator algebras with approximate identities, {I}},
   JOURNAL = {Pacific J. Math.},
  FJOURNAL = {Pacific Journal of Mathematics},
    VOLUME = {303},
      YEAR = {2019},
    NUMBER = {2},
     PAGES = {401--457},
}

@article{El25,
      AUTHOR = {Elki\ae r, E. M.},
     TITLE = {Symmetrized pseudofunction algebras from
              {$L^p$}-representations and amenability of locally compact
              groups},
   JOURNAL = {Expo. Math.},
  FJOURNAL = {Expositiones Mathematicae},
    VOLUME = {43},
      YEAR = {2025},
    NUMBER = {4},
     PAGES = {Paper No. 125685, 20},
}

@misc{CP19,
      title={Simplicity of reduced group Banach algebras}, 
      author={N. C. Phillips},
      year={2019},
      eprint={1909.11278},
      archivePrefix={arXiv},
      primaryClass={math.FA},
}

@misc{DeFaPa25,
      title={Twisted crossed products of Banach algebras}, 
      author={A. Delfín and C. Farsi and J. Packer},
      year={2025},
      eprint={2509.24106},
      archivePrefix={arXiv},
      primaryClass={math.FA},
}

@misc{DdJW11,
      title={Crossed products of Banach algebras. I}, 
      author={S. Dirksen and M. de Jeu and M. Wortel},
      year={2011},
      eprint={1104.5151},
      archivePrefix={arXiv},
      primaryClass={math.FA},
}

@article {De25,
    AUTHOR = {Delf\'in, A.},
     TITLE = {{$L^p$}-modules and {$L^p$}-correspondences},
   JOURNAL = {Banach J. Math. Anal.},
  FJOURNAL = {Banach Journal of Mathematical Analysis},
    VOLUME = {20},
      YEAR = {2026},
    NUMBER = {1},
     PAGES = {Paper No. 13},
}

@article {Ga21,
    AUTHOR = {Gardella, E.},
     TITLE = {A modern look at algebras of operators on {$L^p$}-spaces},
   JOURNAL = {Expo. Math.},
  FJOURNAL = {Expositiones Mathematicae},
    VOLUME = {39},
      YEAR = {2021},
    NUMBER = {3},
     PAGES = {420--453},
}

@article {ChGaTh24,
    AUTHOR = {Choi, Y. and Gardella, E. and Thiel, H.},
     TITLE = {Rigidity results for {$L^p$}-operator algebras and
              applications},
   JOURNAL = {Adv. Math.},
  FJOURNAL = {Advances in Mathematics},
    VOLUME = {452},
      YEAR = {2024},
     PAGES = {Paper No. 109747, 47},
}

@book {Da00,
    AUTHOR = {Dales, H. G.},
     TITLE = {Banach algebras and automatic continuity},
    SERIES = {London Mathematical Society Monographs. New Series},
    VOLUME = {24},
      NOTE = {Oxford Science Publications},
 PUBLISHER = {The Clarendon Press, Oxford University Press, New York},
      YEAR = {2000},
     PAGES = {xviii+907},
}

@article {AuOr22,
    AUTHOR = {Austad, A. and Ortega, E.},
     TITLE = {Groupoids and {H}ermitian {B}anach {$*$}-algebras},
   JOURNAL = {Internat. J. Math.},
  FJOURNAL = {International Journal of Mathematics},
    VOLUME = {33},
      YEAR = {2022},
    NUMBER = {14},
     PAGES = {Paper No. 2250090, 25},
}

@article {GaTh15,
    AUTHOR = {Gardella, E. and Thiel, H.},
     TITLE = {Banach algebras generated by an invertible isometry of an
              {$L^p$}-space},
   JOURNAL = {J. Funct. Anal.},
  FJOURNAL = {Journal of Functional Analysis},
    VOLUME = {269},
      YEAR = {2015},
    NUMBER = {6},
     PAGES = {1796--1839},
}

@article {BaKw25,
    AUTHOR = {Bardadyn, K. and Kwa\'sniewski, B.},
     TITLE = {Topologically free actions and ideals in twisted {B}anach
              algebra crossed products},
   JOURNAL = {Proc. Roy. Soc. Edinburgh Sect. A},
  FJOURNAL = {Proceedings of the Royal Society of Edinburgh. Section A.
              Mathematics},
    VOLUME = {156},
      YEAR = {2026},
    NUMBER = {1},
     PAGES = {157--187},
}

@misc{CP13,
      title={Crossed products of $L^p$ operator algebras and the K-theory of Cuntz algebras on $L^p$ spaces}, 
      author={N. C. Phillips},
      year={2013},
      eprint={1309.6406},
      archivePrefix={arXiv},
      primaryClass={math.FA}
}

@article {Da78,
    AUTHOR = {Dales, H. G.},
     TITLE = {Automatic continuity: a survey},
   JOURNAL = {Bull. London Math. Soc.},
  FJOURNAL = {The Bulletin of the London Mathematical Society},
    VOLUME = {10},
      YEAR = {1978},
    NUMBER = {2},
     PAGES = {129--183},
}

@article {Je77,
    AUTHOR = {Jewell, N. P.},
     TITLE = {Continuity of module and higher derivations},
   JOURNAL = {Pacific J. Math.},
  FJOURNAL = {Pacific Journal of Mathematics},
    VOLUME = {68},
      YEAR = {1977},
    NUMBER = {1},
     PAGES = {91--98},
}

@article {Wi92,
    AUTHOR = {Willis, G.},
     TITLE = {The continuity of derivations from group algebras:
              factorizable and connected groups},
   JOURNAL = {J. Austral. Math. Soc. Ser. A},
  FJOURNAL = {Australian Mathematical Society. Journal. Series A. Pure
              Mathematics and Statistics},
    VOLUME = {52},
      YEAR = {1992},
    NUMBER = {2},
     PAGES = {185--204},
}

@article {flores2025notefinitedimensionalquotientsproblem,
    AUTHOR = {Flores, F. I.},
     TITLE = {A note on finite-dimensional quotients and the problem of
              automatic continuity for twisted convolution algebras},
   JOURNAL = {Cubo},
  FJOURNAL = {Cubo. A Mathematical Journal},
    VOLUME = {27},
      YEAR = {2025},
    NUMBER = {2},
     PAGES = {329--341},
}

\bigskip
\bigskip
ADDRESS

\smallskip
Felipe I. Flores

Department of Mathematics, University of Virginia,

114 Kerchof Hall. 141 Cabell Dr,

Charlottesville, Virginia, United States

E-mail: hmy3tf@virginia.edu

\end{document}